\documentclass[reqno,12pt]{amsart} %
\usepackage{amsmath,amstext, amsthm, amssymb,amsfonts}
\usepackage{fancybox,color,soul}

 \usepackage[top=1 in,bottom=1in, left=1 in, right= 1 in]{geometry}

\numberwithin{equation}{section}

\def\RR{{\mathbb R}}
\def\CC{{\mathbb C}}

\def\NN{{\mathbb N}}

\newcommand{\la}{\lambda}

\newtheorem{theorem}{Theorem}[section]
\newtheorem{proposition}{Proposition}[section]
\newtheorem{lemma}{Lemma}[section]
\newtheorem{corollary}{Corollary}[section]
\theoremstyle{definition}
\newtheorem{definition}{Definition}[section]
\newtheorem{remark}{Remark}[section]

\newcount\mins \newcount\hours \hours=\time \mins=\time
\divide\hours by60 \multiply\hours by 60 \advance\mins by-\hours
\divide\hours by60

\def\now{ \ifnum\hours>11 \ifnum\hours>12 \advance\hours by
-12 \fi \number\hours:\ifnum\mins<10 0\fi \number\mins\ pm,\ \else
\ifnum\hours=0 \hours=12 \fi \number\hours:\ifnum\mins<10 0\fi
\number\mins\ am,\ \fi}

\newcommand{\V}{\mathbb{V}}

\newcommand{\CSK}{Cauchy-Stieltjes Kernel (CSK)\ \renewcommand{\CSK}{CSK\ }}
\newcommand{\FEF}{Free Exponential  (FE)\ \renewcommand{\FEF}{FE\ }}


\title[CSK with polynomial variance functions]{Cauchy-Stieltjes families with polynomial variance functions and generalized orthogonality}

\author{
W{\l}odzimierz  Bryc}
 \address{
Department of Mathematical Sciences\\
University of Cincinnati\\
PO Box 210025\\
Cincinnati, OH 45221--0025, USA} \email{Wlodzimierz.Bryc@uc.edu}

\author{Raouf Fakhfakh}
\address{  Mathematics Department, College of Science and Arts in Gurayat, Jouf University, Gurayat, Saudi Arabia  \&  Laboratory of Probability and Statistics,
Sfax University,  Sfax, Tunisia}
\email{fakhfakh.raouf@gmail.com}

\author{Wojciech M{\l}otkowski}
\address{
 Mathematical Institute\\ University of Wroclaw,
  Pl. Grunwaldzki 2/4 50-384 Wroc\l aw, Poland }
\email{mlotkow@math.uni.wroc.pl}


\date{\today}
\begin{document}

\maketitle

\begin{abstract}
This paper studies variance functions of \CSK families generated by compactly supported centered probability measures.
  We describe several operations that allow us to construct additional  variance functions from known ones. %
  We construct a class of examples which   exhausts all  cubic variance functions, and provide   examples of polynomial variance functions of arbitrary degree.
  We  also relate   \CSK families with polynomial variance functions to generalized orthogonality.

Our main results are  stated solely in terms of classical probability;  some proofs rely on analytic machinery of free probability.
 \end{abstract}

\section{Introduction and main results}

The \CSK families of probability measures  were  introduced in \cite{Bryc-06-08} and extended to
non-compact setting in \cite{Bryc-Hassairi-09}. The constructive approach adopted in these papers   is based on an idea of kernel family
from an unpublished manuscript \cite{Wesolowski90}. The construction emphasizes
analogies with exponential families, using  the Cauchy-Stieltjes kernel $1/(1-\theta
x)$ instead of the exponential kernel $\exp({\theta x})$, and establishing parametrization by the mean. Kernels of the form $h(x\theta)$, including   $1/(1-\theta
x)^a$, appear also in   \cite{Kubo-Kuo-2007} and the references cited therein.

After re-parametrization by the mean, \CSK families  are also a special case $q=0$ of the $q$-exponential families from
\cite{Bryc-Ismail-05}. %
The non-constructive definition  from \cite[Section 4]{Bryc-Ismail-05} is most convenient  for our purposes, as it
emphasizes the role of the pseudo-variance function, which appears directly in the definition.
 \begin{definition}
 The \textit{\CSK family} with a \textit{pseudo-variance function} $\V$ generated by a  compactly supported  non-degenerate  probability measure $\nu$ is a family of probability measures
\begin{equation*} \label{F(V)}
\left\{ Q_m(dx):=f(x,m)\nu(dx): \;m\in(m_-,m_+)\right\},
\end{equation*}
where
\begin{equation}
  \label{f(x,m)}
f(x,m):=\begin{cases}
\frac{\V(m)}{\V(m)+m(m-x)} & m\ne 0;\\
1 & m=0, \V(0)\ne 0; \\
\frac{\V'(0)}{\V'(0)-x} & m=0, \V(0)=0.
\end{cases}
\end{equation}
\end{definition}

The interval $(m_-,m_+)$ is sometimes called the domain of means, but it will not play a major role here.
 We will only assume that $0\in(m_-,m_+)$ and $\V(0)\ne 0$. Then \eqref{f(x,m)} is the solution of the difference equation
\begin{equation}
  \label{DDf} %
  \frac{f(x,m)-f(x,0)}{m}=\frac{x-m}{\V(m)} f(x,m), \quad f(x,0)=1,
\end{equation}
which is a discrete analog of the  differential equation  for exponential families noted in \cite[Theorem 2]{Wedderburn74} (see
also \cite[Section 5]{diBuc:loeb} and \cite{Bryc-Ismail-05}).

It is known that measure $\nu$, if it exists, is uniquely determined (up to the mean) by $\V$, see \cite{Bryc-06-08}.
It is also known that any non-degenerate compactly supported probability measure  $\nu$ gives rise to a unique (real analytic) function $\V$, which we will sometimes denote by $\V_\nu$.  On the other hand, not every function $\V$ can appear as a  pseudo-variance
function. The question of  determining whether
a given class of functions $\V$ corresponds to some measures $\nu$ generated a sizeable literature both for the exponential and more recently for the \CSK families.
 In the theory of exponential families, all quadratic variance functions were determined in \cite{Ism:May}
 and  in \cite{Mor}. All cubic variance functions up to affine transformations are described in
\cite{Let:Mor}. Ref.  \cite{hassairi2004characterization}  characterizes cubic variance
functions by  generalized orthogonality. Numerous non-polynomial variance functions have also been studied, see \cite{Let}; see also \cite[Section 2]{Bryc-Ismail-05}.

The literature  about the variance functions of the \CSK families is less comprehensive.  \CSK families with quadratic variance
functions were determined   in \cite{Bryc-06-08,Bryc-Ismail-05}, see also \cite{Fakhfakh2017}.
  Cubic (pseudo)~variance functions with $\V(0)=0$ have been studied in
\cite{Bryc-Hassairi-09} and they correspond to measures without first moment.  In contrast to exponential families,  \CSK
families are not invariant under translation, so cubic variance functions with  $\V(0)\ne 0$ cannot be reduced to the case
studied in \cite{Bryc-Hassairi-09} and require separate investigation. This paper is devoted solely to the case $\V(0)\ne0$.

We now recall some formulas and  assumptions that we will rely upon.
It is   known
(see \cite[Proposition 3.1]{Bryc-Hassairi-09} or \cite[(3.4)]{Bryc-Ismail-05}) that for $m\ne 0$
\begin{equation}
  \label{mQ}
  \int x Q_m(dx)=m,
\end{equation}
so   family $\{Q_m:m\in(m_-,m_+)\}$ is indeed parameterized by the mean. One can show that if $\nu$ has all moments, $0\in
(m_-,m_+)$, and $\V(0)\ne 0$, then \eqref{mQ} extends by continuity to
\begin{equation}
  \label{m_0}
  \int x\nu(dx)=0.
\end{equation}
We will simply assume \eqref{m_0}. It is then known, and easy to check, that the pseudo-variance
function that appears in \eqref{f(x,m)} is indeed the variance function,
\begin{equation}
  \label{QVar}
  \V(m)=\int (x-m)^2Q_m(dx),
\end{equation}
 see  \cite[Proposition 3.2]{Bryc-Hassairi-09}, and  \cite[(3.4)]{Bryc-Ismail-05} where a more general case was considered.

Denote by  $\mathcal{V}$ the class of variance functions corresponding to
probability measures $\nu$ such that $\nu$ is compactly supported, centered: $\int x\nu(dx)=0$, with variance  $\int
x^2\nu(dx)=1$, so that $\V_{\nu}(0)=1$.
Denote by $\mathcal{V}_{\infty}$ the class of those $\V\in\mathcal{V}$ that  the function
 $m\mapsto \V(cm)$  is in $\mathcal{V}$ for every real $c$.

We begin with some algebraic operations that allow to build new variance functions from known ones. %
(Here we write $\V(m)$ for a function, not its value.)
\begin{theorem} \label{Prop:operations}
Assume that $\V(m)\in\mathcal{V}$, $\V_1(m), \V_2(m)\in\mathcal{V}_{\infty}$ and
 $c\ge1$.
Then
\begin{enumerate}
  \item \label{P-W-3} $\V(m/c)\in\mathcal{V}$;
\item\label{P-W-2}  $\V(m)+a m\in\mathcal{V}$  and   $\V_1(m)+a m\in\mathcal{V}_\infty$ for any $a\in\mathbb{R}$;
\item\label{P-W-4}  $\V_1(m)+\V_2(m)-1\in\mathcal{V}_{\infty}$ and
$c \V_1(m)-c+1\in\mathcal{V}_{\infty}$;
\item \label{P:V-z^2} $\V_1(m)-m^2\in\mathcal{V}$;
\item \label{P:V+z^2} $\V(m)+m^2\in\mathcal{V}_\infty$.
\end{enumerate}
\end{theorem}
The proof of this theorem appears in Section \ref{ProofofT1.3}.
\begin{corollary}
The map $\V(m)\mapsto \V(m)-m^2$ is a bijection
of $\mathcal{V}_{\infty}$ onto $\mathcal{V}$.
\end{corollary}

Next, we  describe  %
the  class of cubic variance functions. %
\begin{theorem}\label{prop:Wojtek} Fix $a,b,c\in\mathbb{R}$. A cubic function $\V(m)=1+am+bm^2 +cm^3$
 is in  $\mathcal{V}$ if and only if $(b+1)^3\geq 27 c^2$.
Furthermore,  $\V$ is in $\mathcal{V}_\infty$ if and only if $b^3\geq 27 c^2$.
\end{theorem}
The proof of this theorem appears in Section \ref{ProofofT1.2}.

Our final result  relates polynomial variance functions for a \CSK family to generalized orthogonality. Suppose $\{P_n(x):
n=0,1,2\dots\}$ is a family of real polynomials, indexed by their degree $n$ with $P_0(x)=1$; it is
sometimes convenient to set $P_k(x)=0$ for $k<0$.

There is a substantial literature on generalized orthogonality and finite-step recursions for polynomials.
We  introduce  the following generalized orthogonality condition.

\begin{definition}\label{Def-d-orth}   Fix $d\in\NN$ and a probability measure $\nu$ with moments of all orders.
 We say that polynomials  $\{P_n\}$ are \textit{$(\nu;d)$-orthogonal}
 if $\int P_n(x) \nu(dx)=0$ for   all $n\geq 1$, and
$$\int P_n(x)P_k(x)\nu(dx)=0 \mbox{ for   all $n\geq 2+(k-1)d$,\quad $k=1,2\dots$.}$$
\end{definition}

  It is clear that  for measures with infinite support,  $(\nu;1)$-orthogonality   is just the standard
orthogonality.
 For $d=2$, we recover  \cite[Definition 3.1]{hassairi2004characterization}.
 The concept of
$d$-orthogonality introduced in \cite{Iseghem1987approx}  is different  as even for $d=1$ it has no
 positivity requirements for the functional/measure.  When $d>2$, condition of  pseudo-orthogonality in
\cite{kokonendji2005characterizations,kokonendji2005d} is also different.
It is somewhat interesting to note that various concepts of generalized orthogonality are related to $(d+2)$-step recursions for the polynomials,
so the distinctions
sometimes rely on minute technicalities, see the paragraph above Corollary \ref{Cor-favard3}.

The following result  is  a generalization of \cite[Theorem 3.2]{Fakhfakh2017} to $d>1$, and   a \CSK-version of \cite[Theorem
3.1]{hassairi2004characterization} when $d=2$.
\begin{theorem}\label{T:poly}
Suppose that $\V$ is a variance function of a \CSK family generated by
 a non-degenerate compactly supported probability measure $\nu$  with mean $0$ and variance $1$.
Consider the family of polynomials $\{P_n(x)\}$ with generating function
  \begin{equation}\label{f2P}
f(x,m)=\sum_{n=0}^\infty  P_n(x) m^n,
\end{equation} where $f(x,m)$ is given by \eqref{f(x,m)}.
Then the following statements are equivalent:
\begin{enumerate}
  \item\label{i} $\V(m)=1+\sum_{k=1}^{d+1}a_k m^k$ is a polynomial of degree  at most $d+1$;
\item\label{iv} There exist constants $\{b_k:k=1,\dots, d+1\}$   such that polynomials $\{P_n\}$ satisfy
 recursion
\begin{equation}
  \label{CCRd}
  x P_n(x)= P_{n+1}(x)+\sum_{k=1}^{(d+1)\wedge n}b_k P_{n+1-k}(x),\quad n\geq 1
\end{equation}
 with   initial conditions $P_0(x)=1$, $P_1(x)=x$.

  \item\label{ii} Polynomials $\{P_n(x)\}$ are $(\nu;d)$-orthogonal.
 \end{enumerate}
\end{theorem}
 Note that
the upper limit of the sum on the right hand side of  \eqref{CCRd} is $d+1$ under the   convention that $P_{k}(x)=0$ for $k<0$, and
 that Proposition \ref{P:Wojtek-1} below provides examples of polynomial variance functions of arbitrarily high degree. The proof of Theorem \ref{T:poly} appears in Section \ref{ProofofT1.5}.

The paper is organized as follows.
In Section \ref{Sec:Var} we introduce  free probability notation and use it to prove  the first two theorems. %
We also include
some additional examples of variance functions.
Section \ref{Sec:proofs}  is independent of Section \ref{Sec:Var}  and discusses results on polynomials that imply Theorem \ref{T:poly}. %
In Section~\ref{Sec:conclusions} we provide a combinatorial example
involving sequences A001764, A098746 and A106228 from
 OEIS~\cite{sloane2003line}.
We also discuss generating functions and sharpness of some results.

\section{Variance functions and free probability}\label{Sec:Var}

Recall that a \textit{dilation} $D_t(\nu)$ of a probability measure $\nu$ by a non-zero real number $t$ is a measure $\mu(U)=\nu(U/t)$. $D_{-1}(\nu)$ is called the \textit{reflection of $\nu$}.
In the language of probability theory, dilation changes the law of random variable $X$ to the law of $t X$.

\subsection{Notation from free probability}

For a probability measure $\mu$ on $\mathbb{R}$ we put:
$$
M_{\mu}(z):=\int \frac{\mu(dx)}{1-xz},\quad
G_\mu(z):=\int \frac{\mu(dx)}{z-x},\quad
F_\mu(z):=1/G_\mu(z),
$$
$M_{\mu}$ is called the \textit{moment generating function},
$G_\mu(z)$ is the \textit{Cauchy-Stieltjes transform}.
The \textit{free $R$-transform} can be defined by the equation
\begin{equation}\label{R2M}
R_\mu(z M_\mu(z))+1=M_\mu(z).
\end{equation}
The coefficients $\kappa_n(\mu)$ in the Taylor expansion $R_\mu(z)=\sum_{n=1}^\infty\kappa_n(\mu) z^n$
are called \textit{free cumulants}.
We will also use
\begin{equation}
r_\mu(z):=R_\mu(z)/z.
\end{equation}
Equation \eqref{R2M} can also be written as
\begin{equation}\label{littlertransformrelation}
z M_{\mu}(z) r_{\mu}(zM_{\mu}(z))=M_{\mu}(z)-1.
\end{equation}
Note that for the dilated measure we have
\begin{equation}\label{dilationformulas}
M_{D_t(\mu)}(z)=M_{\mu}(tz),\quad
R_{D_t(\mu)}(z)=R_{\mu}(tz),\quad
r_{D_t(\mu)}(z)=t r_{\mu}(tz).
\end{equation}
If $|z|\ne0$ is small enough then
\begin{equation}   \label{FGr} G_{\mu}\left(r_{\mu}(z)+\frac{1}{z}\right)=z,\qquad
F_{\mu}\left(r_{\mu}(z)+\frac{1}{z}\right)=\frac{1}{z}.
\end{equation}

The sum of two $R$-transforms is an $R$-transform and  defines
the \textit{free additive convolution} of measures $\mu\boxplus \nu$   by
$R_{\mu\boxplus \nu}(z)=R_\mu(z)+R_\nu(z)$. For any real $t\geq 1$,
it is known that $t R_\mu(z)$ is an $R$-transform and defines
\textit{additive free convolution power} $\mu^{\boxplus t}$ (see \cite{Nica-Speicher}).

Probability measure $\mu$ is called \textit{$\boxplus$-infinitely  divisible}
if its free convolution power $\mu^{\boxplus t}$ is well defined for all real $t>0$.
If $\lambda\ne0$ then $\mu$ is $\boxplus$-infinitely  divisible
if and only if $D_{\lambda}(\mu)$ is $\boxplus$-infinitely  divisible.

It is known, see   \cite{bercovici2000free,Hiai-Petz00}
that a compactly supported $\mu$ with the first moment $m_0=\int x\mu(dx)$
is $\boxplus$-infinitely divisible if and only if there exists a compactly supported
finite measure $\omega$ on $\RR$ such that $\omega(\RR)=\int (x-m_0)^2 \mu(dx)$ and
\begin{equation*}
  r_\mu(z)=m_0+ z\int \frac{\omega(dx)}{1-z x}.
\end{equation*}
In particular, if $\nu$ is a generating measure of a \CSK family, then under our moment assumptions,   $\nu$ is free-infinitely
divisible if an only if there is a compactly supported probability measure $\omega$ such that
\begin{equation}
  \label{M2r}
  r_\nu(z)=zM_\omega(z).
\end{equation}

 For a probability measure $\mu\ne \delta_0$ with support in $[0,\infty)$, the \textit{$S$-transform}
  is defined by
 \begin{equation}\label{M2S}
   R_\mu(z S_\mu(z))=z\quad \mbox{or}\quad
   M_\mu\left(\frac{z}{1+z}S_\mu(z)\right)=1+z,
 \end{equation}
see e.g.  \cite[(5)]{HM-2011}. Note that in particular $\int x \mu(dx)=1/S_\mu(0)$.

The product of $S$-transforms is an $S$-transform and defines
the \textit{multiplicative free convolution} $\mu_1\boxtimes\mu_2$ by
$S_{\mu_1\boxtimes\mu_2}(z)=S_{\mu_1}(z) S_{\mu_2}(z)$.
\textit{Multiplicative free convolution powers} $\mu^{\boxtimes p}$ are defined at least for all $p\geq 1$ (see \cite[Theorem 2.17]{belinschi2006complex})
by $S_{\mu^{\boxtimes p}}(z)=S_\mu(z)^p$.

The \textit{Marchenko-Pastur measure} with parameter $\lambda>0$:
$$
\pi_\la(dx)=(1-\la)^+\delta_0+\frac{\sqrt{4 \la - (x-1-\la)^2}}{2\pi x}1_{x\in [(1-\sqrt{\la})^2, (1+\sqrt{\la})^2]}dx.
$$
plays in free probability the role of the Poisson distribution, see~\cite{nicaspeicher2006}.
Since $S_{\pi_\la}(z)=\frac{1}{\la+z}$, we have %
\begin{equation*}
  S_{D_{b}(\pi_{1/b})}(z)=\frac{1}{1+b z}.
\end{equation*}

It is known, see
\cite[Section 2]{HM-2011}, \cite[Theorem 1.2]{arizmendi2016classical} and \cite{banica2011free,mlotkowski2010fuss}, that
if $p>0$, $b>0$ then $\mu=\left(D_{b}(\pi_{1/b})\right)^{\boxtimes p}$
 exists if and only if $\max\{p,1/b\}\geq 1$.
 This measure $\mu$ has compact support in $[0,\infty)$ and its $S$-transform equals:
\begin{equation}\label{Poisson-S}
  S_{\mu}(z)=\frac{1}{(1+b z)^p}.
\end{equation}

For additional details and background on free probability we refer to
\cite{nicaspeicher2006,voiculescu1992free}.
\subsection{Formulas for variance functions}
A variance function $\V$ of a \CSK family generated by a compactly supported  centered probability measure
$\nu\ne \delta_0$ is real-analytic at $m=0$, so it extends to the analytic mapping
 $z\mapsto \V(z)$  on an open disk near $z=0$.
Our assumptions
  on  the  first two moments of $\nu$ imply that $r_\nu(z)=z+\kappa_3(\nu)z^2+\dots$  is invertible near $z=0$ and   its composition inverse is
${z}/{\V_\nu(z)}$ (\cite[Theorem 3.3]{Bryc-06-08}),
so that
\begin{equation}\label{r-V}
  r_{\nu}(z)=z \V_{\nu}(r_{\nu}(z)).
\end{equation}
Replacing $z$ by
$ z/\V(z)$, from equation \eqref{FGr} we get
\begin{equation}\label{F2V}
  F_\nu\left(z+\frac{\V_\nu(z)}{z}\right)=\frac{\V_\nu(z)}{z}.
\end{equation}
(This was first noted in \cite[(4.4)]{Bryc-Ismail-05} and exploited  in
\cite{Bryc-06-08,Bryc-Hassairi-09,Bryc-Raouf-Hassairi-14}.)

The following result is known but we prove it for completeness.
\begin{lemma}\label{P-W-0}
  If $z\mapsto \V(z)$ is a variance function then so is $z\mapsto \V(-z)$.
\end{lemma}
\begin{proof}
  Put $\nu_{-}:=D_{-1}(\nu)$. Then, by (\ref{dilationformulas}), $r_{\nu_{-}}(z)=-r_{\nu}(-z)$
and from (\ref{r-V}) we have $\V_{\nu_{-}}(z)=\V_{\nu}(-z)$.
\end{proof}
The following relates   class $\mathcal{V}_{\infty}$  of variance functions to free probability.
\begin{proposition}
  \label{P:W-3a}
If $\V=\V_{\nu}$ then  $\nu^{\boxplus
\lambda^2}$ exists if and only if $\V(z/\la)\in\mathcal{V}$. In particular,    $\mathcal{V}_{\infty}$ is the class of those $\V_{\nu}\in\mathcal{V}$ that $\nu$
is $\boxplus$-infinitely divisible.

\end{proposition}

\begin{proof} %
Suppose $\nu^{\boxplus \lambda^2}$ exists and denote $\nu_{\lambda}:=D_{1/\lambda}(\nu^{\boxplus \lambda^2})$.
Then, by (\ref{dilationformulas}), we have $r_{\nu_{\lambda}}(z)=\lambda r_{\nu}(z/\lambda)$ and
$$
\frac{r_{\nu_{\lambda}}(z)}{\V(r_{\nu_{\lambda}}(z)/\lambda )}=
\frac{\lambda r_{\nu}(z/\lambda)}{\V(r_{\nu}(z/\lambda))}
=\lambda\frac{z}{\lambda}=z,
$$
which proves that $\V_{\nu_{\lambda}}(z)=\V(z/\lambda)$.
Conversely, from the first equality, if $\V(z/\lambda)$ is a variance function of some $\nu_\la$ then $r_{\nu_\la}(z)/\la=r_\nu(z/\la)$ so $\nu^{\boxplus \lambda^2} =D_{\lambda}(\nu_{ \lambda})$ exists.

In particular, from Lemma \ref{P-W-0} we see that $\V\in\mathcal{V}_\infty$ if and only if $\nu$ is $\boxplus$-infinitely divisible.
\end{proof}

From \eqref{r-V},  \eqref{M2r} and \eqref{R2M}
we get
the following.
\begin{proposition}\label{P:W-1}
A function $\V(z)$ belongs to $\mathcal{V}_{\infty}$ if and only if there is a compactly supported probability measure
$\omega$ on $\mathbb{R}$ such that $\V(z)=1+R_{\omega}(z)$.
\end{proposition}

 We remark that the perturbation theorem
in \cite{bercovici1995superconvergence}    generates  a large number of
implicit examples
of variance functions in $\mathcal{V}_\infty$. In particular,
for every $d\geq 3$ there is a $\delta>0$ such that
$\V(z)=1+z^2+\sum_{k=3}^d c_k z^k$
is  in $\mathcal{V}_\infty$ when  $\max_k|c_k|<\delta$.
Corollary 2.5 in \cite{chistyakov2011characterization} yields explicit
characterization
of such variance functions for $d=4$.

\subsection{Proof of Theorem \ref{Prop:operations}}\label{ProofofT1.3}
We need the  following lemma that we shall use with $\alpha=1$, $\beta=0$.
\begin{lemma}\label{L:inverse}
If $\omega$ is a probability distribution on $\mathbb{R}$, $\alpha>0$, $\beta\in\mathbb{R}$ then there exists a non-degenerate
probability
distribution $\mu$ such that
\begin{equation}\label{lemat}
M_{\mu}(z)=\frac{1}{1-\beta z-\alpha z^2 M_{\omega}(z)}.
\end{equation}

Conversely, if $\mu$ is a probability measure with moments
\[
\int x \mu(dx)=\beta,\quad
\int(x-\beta)^2\mu(dx)=\alpha>0
\]
then there exists a probability measure $\omega$ such that \eqref{lemat}
holds.
\end{lemma}

\begin{proof}
For the $F$-transform of $\mu$ we have $M_{\mu}(z)=1/(zF_{\mu}(1/z))$, so
 relation \eqref{lemat} becomes
\[
F_{\mu}(z)=z-\beta-\frac{\alpha}{F_{\omega}(z)}.
\]
Now it suffices to apply Proposition~5.2 from \cite{bercovici1993free} (see also \cite[Section 3.3]{Hiai-Petz00}).

To prove the converse, we apply Proposition~5.2 from \cite{bercovici1993free} to analytic function $$F(z):= \frac{\alpha}{z-\beta-F_\mu(z)},$$
which becomes the $F$ transform of a probability measure.
To verify the assumptions in \cite{bercovici1993free} we note that since $\mu$ is non-degenerate we have $\Im F_\mu(z)>\Im z$ (see comments below \cite[Proposition 2.1]{maassen1992addition}). So $F$ maps $\CC_+$ into itself. Series expansion at $\infty$ gives $z-\beta-F_\mu(z)=\alpha/z+o(1/z)$ as $|z|\to \infty$.
\end{proof}

\begin{proof}[Proof of Theorem \ref{Prop:operations}]
Statement
\eqref{P-W-3}
follows  from Proposition \ref{P:W-3a}, as free convolution power $\nu^{\boxplus c^2}$ exists for $c\geq 1$.

\eqref{P-W-2}
Let $G(z)=  M_\nu(1/z)/z$ be the Cauchy-Stieltjes transform of $\nu$ and   $F(z)=1/G(z)$. The  continued fraction expansion for $G$ gives
$$
F(z)=z-b_0-\cfrac{c_0}{z-b_1-\cfrac{c_1}{z-b_2-\frac{c_2}{\ddots}}}
$$
where $b_n,c_n$ are the Jacobi coefficients in the three-step recursion for the monic orthogonal polynomials with
respect to measure $\nu$,
$$
x p_n(x)=p_{n+1}(x)+b_n p_n(x)+c_{n-1}p_{n-1}(x),\quad n\geq 0.
$$
(This can be read out from \cite[Section 2.6]{Ism2005}. The recursion and the continued fraction terminate at $c_N=0$ if $\nu$
is purely atomic with
$N+1$ atoms.)

 Define $F_a(z)=F(z-a)+a$. Then $F_a(z)$ has the same continued fraction expansion with the same coefficients $c_n$, the same coefficient $b_0$,  and for $k\ge1$ coefficient $b_k$ is replaced by $b_k+a$.
 Therefore, by Favard's theorem (the usual version, or a finite version when $c_N=0$; the latter can be
 read out from  the first page of \cite[Section 2.5]{Ism2005}) $F_a(z)$ is the inverse of a Cauchy-Stieltjes transform of a probability  measure
 $\nu_a$. The first two moments of $\nu_a$ are not affected by the change of $b_1,b_2,\dots$, so $\nu_a$ has  mean 0 and variance $1$.

Since $F$ is well defined outside of the support of $\nu$, we have $F(x)>0$ for $x>K$ and $F(x)<0$ for $x<-K$.
So $F_a$ also extends to the real axis far away from $0$, and therefore $\nu_a$ has compact support.
 (This fact is sometimes called Krein's theorem \cite{krein1977markov}, see e.g. \cite[Theorem 3.9]{chistyakov2011characterization}.)

Since $F$ satisfies \eqref{F2V}, function $\V_a(z)=\V(z)+az$ satisfies the same identity with $F_a$
in place of $F$, identifying the variance function.

Suppose now that  $\V\in\mathcal{V}_\infty$. Then $\V(cz)$ is a variance function for any real $c$,  so by the previous reasoning with $a$ replaced by $ac$, we see that
$\V(cz)+ac z=\V_a(cz)$ is in $\mathcal{V}$, i.e., $\V_a\in\mathcal{V}_\infty$.

\eqref{P-W-4} We use Proposition \ref{P:W-1}.
If $\V_1(z)=\V_{\nu_1}(z)=1+R_{\omega_1}(z)$, $\V_2(z)=\V_{\nu_2}(z)=1+R_{\omega_2}(z)$
then
\[
\V_1(z)+\V_2(z)-1=1+R_{\omega_1}(z)+R_{\omega_2}(z)=1+R_{\omega_1\boxplus\omega_2}(z)
\]
and similarly
\[
c \V_1(z)-c+1=1+c R_{\omega_1}(z)=1+R_{{\omega_1}^{\boxplus c}}(z).
\]

\eqref{P:V-z^2}
Let $\mathbb{V}_1=\mathbb{V}_{\nu_1}$ and denote $r_1:=r_{\nu_1}$.
Then $r_1(z)=zM_{\omega}(z)$ for some probability measure $\omega$.
Using Lemma \ref{L:inverse}, let $\nu$ be such a probability measure that $M_{\nu}(z)=1/(1-z r_1(z))$.
It is clear  that   $\nu$ has mean zero and variance 1.
Denote $M_{\nu}=M$, $r_{\nu}:=r$ and put $\widetilde{z}:=z M(z)$. Then,
by (\ref{littlertransformrelation}),
\[
z=\frac{\widetilde{z}}{M(z)}=\frac{\widetilde{z}}{\widetilde{z}r(\widetilde{z})+1}
\]
and
\[
r_1(z)=\frac{M(z)-1}{zM(z)}=r(\widetilde{z}).
\]

Applying these identities to the equality
\[
\frac{r_1(z)}{\mathbb{V}_1(r_1(z))}=z
\]
yields
\[
\frac{r(\widetilde{z})}{\mathbb{V}_1(r(\widetilde{z}))}=\frac{\widetilde{z}}{\widetilde{z}r(\widetilde{z})+1}
\]
or equivalently
\[
\frac{r(\widetilde{z})}{\mathbb{V}_1(r(\widetilde{z}))-r(\widetilde{z})^2}=\widetilde{z},
\]
which proves that $\mathbb{V}_{\nu}(z)=\mathbb{V}_1(z)-z^2$.

\eqref{P:V+z^2} Let $\nu$ be the measure corresponding to $\V$. By the converse part of  Lemma \ref{L:inverse},
there exists a compactly supported probability measure $\omega$ such that $M_{\nu}(z)=1/(1-z^2M_{\omega}(z))$. Then  \eqref{M2r} defines
measure $\nu_1$ with  $r_{\nu_1}(z)=zM_{\omega}(z)$ and the relation $\V(z)=\V_{\nu_1}(z)-z^2$ holds by the  proof of part \eqref{P:V-z^2}.
\end{proof}

\subsection{Proof of Theorem \ref{prop:Wojtek}}\label{ProofofT1.2}
The proof of Proposition \ref{P:Wojtek-1} uses $S$-transforms from \eqref{Poisson-S}.

\begin{lemma}\label{L:V2S}
Suppose that $\mu\ne \delta_0$ is a probability measure with compact support in $[0,\infty)$ and that $S_\mu(0)=1$.   Then
\begin{equation}
  \label{V2S}
 \V(z) =\frac{1+z}{S_\mu(z)}
\end{equation}
is in $\mathcal{V}_\infty$.
\end{lemma}
\begin{proof}
Define
$\omega(dx)=x\mu(dx)$, and note that this is a probability measure since $\omega(\RR)=\int x \mu(dx)=1/S_\mu(0)=1$.  Let $\nu$ be the $\boxplus$-infinitely divisible probability measure  defined by \eqref{M2r}. Then
  $M_\mu(z)= 1+r_\nu(z)$, so \eqref{M2S} gives
   $$r_\nu\left(\frac{z}{1+z} S_\mu(z)\right)=z.$$ Recalling that composition inverse of  $z\mapsto r(z)$   is $z/\V_\nu(z)$, in a neighborhood of $z=0$ we
   get \eqref{V2S}.

\end{proof}

Lemma   \ref{L:V2S} yields a class of variance functions in $\mathcal{V}_\infty$ of the following form.

\begin{lemma}\label{L:V0}

Let $\beta_1,\dots,\beta_d>0$, $p_1,\dots,p_2>0$, with $\max\{p_j,1/\beta_j\}\ge1$. Then the function
\begin{equation}
  \label{V0}
   \mathbb{V}(z)=(1+z)\prod_{j=1}^d(1+\beta_j z)^{p_j}
\end{equation}
is in $\mathcal{V}_\infty$.
\end{lemma}

\begin{proof}
For $j\geq 1$, choose $\mu_j$ with $S_{\mu_j}(z)=(1+\beta_j z)^{-p_j}$, see \cite{arizmendi2016classical,HM-2011}.
Define  $\mu=\mu_1\boxtimes\mu_2\boxtimes\dots\boxtimes \mu_d$  so that
$$
S_\mu(z)=\prod_{j=1}^d (1+\beta_j z)^{-p_j}.
$$
Lemma \ref{L:V2S} ends the proof.
\end{proof}
We will deduce sufficiency in Theorem \ref{prop:Wojtek} from the following general result.

\begin{proposition} \label{P:Wojtek-1}
Assume that $d\in\NN$, $a,b\in\mathbb{R}$, $c>0$,
$b_1,\dots,b_d> 0$, $p_1,\dots,p_d> 0$ and that $\max\{p_j,c/b_j\}\ge1$
for $1\le j\le d$. Put
$$
\V(z)=a z+b z^2+(1+cz)\prod_{j=1}^d(1+b_jz)^{p_j}.
$$
If $b\geq -1$ then  $\V\in \mathcal{V}$. If $b\geq 0$ then  $\V\in \mathcal{V}_\infty$.
 \end{proposition}

In the present paper we are mainly interested in polynomial variance functions,
however here we would like to emphasize that the exponents
$p_j$ do not have to be integers; for example $(1+z)^{\sqrt{2}}$ or
$(1+z)^{3/2}(1+2z)^{3/2}$   are  variance functions in $\mathcal{V}_\infty$.

 \begin{proof}
Put $\V_1(z):=1+a z+b z^2$, $\V_2(z):=(1+cz)\prod_{j=1}^d(1+b_jz)^{p_j}$.
Then $\V_2\in\mathcal{V}_{\infty}$ in view of Lemma~\ref{L:V0} with $\beta_j=b_j/c$.
If $b\ge0$ then $\V_1\in\mathcal{V}_{\infty}$ (see \cite[Theorem 3.2]{Bryc-06-08} and the comments therein)
and consequently $\V(z)=\V_1(z)+\V_2(z)-1\in\mathcal{V}_{\infty}$
by Theorem~\ref{Prop:operations}\eqref{P-W-4},
which proves that $\V\in\mathcal{V}_\infty$ when $b\geq 0$. When $b\geq-1$,  we apply Theorem~\ref{Prop:operations}\eqref{P:V-z^2} to $ z^2+\V_1(z)+\V_2(z)-1\in \mathcal{V}_\infty$.
\end{proof}

We are now ready to prove  Theorem \ref{prop:Wojtek}.

\begin{proof}[Proof of Theorem \ref{prop:Wojtek}]
The case $c=0$ is well understood: $\V\in \mathcal{V}$ if and only if
$b+1\geq 0$ and $\V\in \mathcal{V}_\infty$ if and only if $b\geq 0$, see \cite{Bryc-06-08}.
In view of Lemma~\ref{P-W-0} we can assume that $c>0$.

Applying Proposition~\ref{P:Wojtek-1} with $d=1$, $b_1=c$ and $p_1=2$ we get that
$$az +bz^2+(1+cz)^3=1+(a+3c)z+(b+3 c^2)z^2+c^3 z^3$$
is in $\mathcal{V}$ for any $b\geq -1$, and in
 $\mathcal{V}_\infty$ for any $b\geq 0$, with any real
$a,c$.
Replacing
$a+3c, b+3c^2,c^3$ by $a, b , c$ respectively we get  the sufficient conditions for $\V\in \mathcal{V}$ and for $\V\in \mathcal{V}_\infty$ as stated (recall that $c>0$).

It remains to show that if  $b^3< 27 c^2$, then  $\V(z)=1+az+bz^2 +cz^3$ is not in $\mathcal{V}_\infty$.
  By Theorem \ref{Prop:operations}\eqref{P-W-2}, without loss of generality we may assume $a=0$.

We proceed by contradiction. Suppose $\V\in \mathcal{V}_\infty$. If  $b>0$ then by scaling  we'd get $1+z^2+c z^3\in \mathcal{V}_\infty$ for some (different) $c^2>1/27$.
By Proposition  \ref{P:W-1} this would mean that there exists a compactly supported probability measure $\omega$ with $r_\omega(z)=z+c z^2$,
contradicting \cite[Corollary 2.5]{chistyakov2011characterization} which  says this
 to be  possible if and only if  $c^2\leq 1/27$.

Suppose now that   $b\leq 0$. Then  by Theorem \ref{Prop:operations}\eqref{P-W-4} with $\V_2(z)=1 +|b| z^2$ we'd get  $1+c z^3 \in \mathcal{V}_\infty$.
Since  $c> 0$, we'd be able to rescale  and get, say,
$1+2z^3\in \mathcal{V}_\infty$.
Using Theorem \ref{Prop:operations}\eqref{P-W-4} again, we would get $1+z^2+2 z^3\in \mathcal{V}_\infty$, contradicting
\cite[Corollary 2.5]{chistyakov2011characterization} again.
(In fact, as explained in Remark \ref{R:sharp2} below, $1+2z^3$ is not in  $\mathcal{V}$.)
\end{proof}

\section{Variance functions and polynomials}\label{Sec:proofs}

In general, if $\V$ is (real) analytic at $0$ and $\V(0)\ne 0$, it is easy to see that expansion \eqref{f2P} holds, and its coefficients are polynomials $\{P_n(x)\}$ which
solve the recursion
\begin{equation}
  \label{P-rec}
  x P_n(x)= P_{n-1}(x)+\sum_{k=0}^n\frac{\V^{(k)}(0)}{k!} P_{n+1-k}(x),\; n\geq 0,
\end{equation}
with initial polynomials $P_{-1}(x)=0$ and $P_0(x)=1$.
 (In particular, polynomials $\{P_n\}$ are monic  when $\V(0)=1$.)
 To derive \eqref{P-rec},  multiply \eqref{DDf} by $m\V(m)$, expand $\V$ into the power series at $m=0$, expand $f(x,m)$ into power series (recall that $\V(0)\ne 0$),
  and compare the coefficients at the powers of $m$.

We therefore consider a slightly more general recursions than \eqref{CCRd}.
  Suppose that polynomials $\{P_n\}$ satisfy the recursion \begin{equation}
  \label{P-rec-a}
  x P_n(x)= P_{n-1}(x)+\sum_{k=0}^na_kP_{n+1-k}(x),\; n\geq 0
\end{equation}
with $a_0\ne 0$ and initial polynomials $P_{-1}(x)=0$ and $P_0(x)=1$.

 \begin{proposition}
  \label{Lem-intP} Suppose that there are $A, R>0$ such that $|a_k|\leq A R^k$ for all $k=0,1,\dots$.
  Define $\V(z)=\sum_{k=0}^\infty a_k z^k$ for $|z|<1/R$.
  \begin{enumerate}
\item  If polynomials $\{P_n\}$ satisfy  recursion \eqref{P-rec-a} then
\begin{equation}\label{referree-fix}
\sum_{n=0}^\infty P_n(x) z^n=\frac{\V(z)}{\V(z)+z(z-x)}
\end{equation}

  and the series converges uniformly over $x\in K$ for any compact set $K\subset\RR$. That is, there is  $r>0$ that does not depend on $x\in K$ such
  that the series   converges uniformly over $x\in K$ for all $|z|<r$.
  \item If polynomials $\{P_n\}$ satisfy  recursion \eqref{P-rec-a}  and there is a
  non-degenerate compactly supported centered
  probability
  measure $\nu$ such that $\int P_n(x)\nu(dx)=0$ for all
  $n\geq 1$, then $a_0=\V(0)>0$,
  and  $\V(\cdot)$ is the variance function of a \CSK family generated by $\nu$.

\item If $\V(\cdot)$ is a variance function of a \CSK family generated by  a non-degenerate centered compactly supported probability measure $\nu$ and $\{P_n\}$
are polynomials from \eqref{f2P}   then  $\int P_n(x)\nu(dx)=0$ for $n\geq 1$.

  \end{enumerate}
\end{proposition}

\begin{proof} (i) Since $P_0(x)=1$, without loss of generality, we may assume that $A=1$.
Let $M=\sup_{x\in K}|x|$.  Choose $C>R$ such that
\begin{equation}
  \label{C(M,R)}
  \frac{M}{C}+\frac{1}{C^2}+\frac{R}{C-R}\leq |a_0|.
\end{equation}
We now check by induction that with this choice of $C$ we have
\begin{equation}
  \label{supP}
  \sup_{x\in K}|P_n(x)|\leq C^n \mbox{ for all $n\geq 0$}.
\end{equation}
Clearly, $|P_0(x)|\leq 1\leq C^0$ and $\sup_{x\in K} |P_1(x)|= \sup_{x\in K} |x/a_0|\leq M/|a_0|\leq C$. Suppose that $N\geq 1$ is such that \eqref{supP} holds for
all $P_n$ with $n\leq N$. From \eqref{P-rec-a} we see that
\begin{multline*}
|a_0|\sup_{x\in K} |P_{N+1}(x)|\leq \sup_{x\in K}|x P_N(x)|+\sup_{x\in K}| P_{N-1}(x)|+\sum_{k=1}^N R^k \sup_{x\in K}|P_{N+1-k}(x)|
\\
\leq M C^N+C^{N-1}+C^{N+1}\sum_{k=1}^N(\tfrac{R}{C})^k\leq C^{N+1} \left(\tfrac{M}{C}+\tfrac{1}{C^2}+\tfrac{R}{C-R}\right)\leq |a_0|
C^{N+1}
\end{multline*}
by  \eqref{C(M,R)}. This proves \eqref{supP} by induction.

From \eqref{supP} it is clear that with $r=1/C$   the series  \eqref{f2P} converges uniformly over $x\in K$ for all (complex) $|m|<r$.

To identify the limit, denote the sum of the series by $\varphi(x,z)$. Multiplying \eqref{P-rec-a} by $z^n\ne 0$ and summing over $n$, we get
  \begin{equation}\label{phi-eq}
    x\varphi(x,z)=z \varphi(x,z)+\frac{1}{z}\sum_{n=0}^\infty  \sum_{k=0}^nz^ka_k z^{n+1-k} P_{n+1-k}(x).
  \end{equation}
  Changing the order of summation,
 \begin{multline*}\sum_{n=0}^\infty \sum_{k=0}^nz^ka_k z^{n+1-k} P_{n+1-k}(x) =\sum_{k=0}^\infty z^ka_k \sum_{n=k}^\infty  z^{n+1-k} P_{n+1-k}(x)\\=
 \sum_{k=0}^\infty z^ka_k  \left(\varphi(x,z)-1\right) = \V(z) \left(\varphi(x,z)-1\right).
 \end{multline*}
Inserting this into \eqref{phi-eq} we see that
$$
x\varphi(x,z)=z \varphi(x,z)+\frac{\V(z)}{z}\left(\varphi(x,z)-1\right).
$$
The solution of this equation is $\varphi(x,z)=\frac{\V(z)}{\V(z)+z(z-x)}$ as claimed.

(ii)
Since polynomial  $P_2(x)=x^2/a_0^2-a_1 x/a_0^2-1/a_0$ integrates to $0$,  and $\int x\nu(dx)=0$ by assumption, we see that $a_0>0$. So
$\V(m)\geq 0$ in some neighborhood of zero and on the support of $\nu$ the generating function $f(x,m)\geq 0$ for $m$ small enough.

 Since  $\int P_n(x)\nu(dx)=0$ for $n\geq 1$,  and by part (i)   series \eqref{referree-fix} converges uniformly on the support of $\nu$,
  integrating   term-by-term we get $\int f(x,m)\nu(dx)=1$, i.e.,
$\V$ is  the variance function of the CSK family generated by $\nu$.

(iii)  Suppose  that $\{P_n\}$ are polynomials from \eqref{f2P} and $\nu(dx)$ has compact support. Then \eqref{P-rec} implies \eqref{P-rec-a} with $a_k=\V^{(k)}(0)/k!$.
Since   $\V(0)\ne 0$ and $\V$ is real analytic, one can find $R>1$  such that $|\V^{(k)}(0)|\leq k! \V(0) R^k$, so the assumption  on the growth of $|a_k|$ is satisfied.
By uniform convergence
for all small enough $m$ we can integrate series \eqref{f2P} term-by-term. We get
$$
1=\int f(x;m)\nu(dx)=1+\sum_{n=1}^\infty m^n \int P_n(x)\nu(dx).
$$
Thus $\int P_n(x)\nu(dx)=0$ for all $n\geq 1$.
\end{proof}

Next, we relate polynomial variance functions to $(\nu;d)$-orthogonality.
\begin{proposition}\label{T:V2d}
Suppose that $\V$ is a variance function of a \CSK family generated by
 a non-degenerate compactly supported probability measure $\nu$  with mean $0$ and variance $1$.
Consider the family of polynomials $\{P_n(x)\}$ with generating function \eqref{f2P}, where $f(x,m)$ is given by \eqref{f(x,m)}. Then
\begin{enumerate}
  \item $\int P_n(x)\nu(dx)=0$ for $n\geq 1$.
  \item If  polynomial $P_2(x)$ is orthogonal in $L_2(\nu)$ to all polynomials $\{P_n(x): n\geq 2+d\}$, then  the family $\{P_n(x)\}$
   is $(\nu;d)$-orthogonal, satisfies recursion \eqref{CCRd}
       and $\V$ is a polynomial of degree at most $d+1$.
\item Conversely, if the variance function $\V$ of a \CSK family generated by measure
$\nu$  is a polynomial of degree at most $d+1$ then the polynomials from expansion \eqref{f2P} are $(\nu;d)$-orthogonal.
\end{enumerate}

\end{proposition}
\begin{proof}
(i) This is included in Proposition \ref{Lem-intP}(iii).

 (ii)  Since $P_2(x)=x^2-\V'(0)x-1$ and $\V(m)$ is given by  \eqref{QVar}, we see that
 \begin{equation}
   \label{P2f1}
   \int P_2(x)f(x,m)\nu(dx)=\V(m)+m^2-\V'(0)m-1.
 \end{equation}
On the other hand due to uniform convergence (Lemma \ref{Lem-intP}), for  all small enough $m$ we can
  integrate  series \eqref{f2P} term by term.  Since  by assumption $\int P_2(x)P_k(x)\nu(dx)=0$ for $k\geq 2+d$, we get
 \begin{multline}
   \label{P2f2}\int P_2(x)f(x,m)\nu(dx)=\int P_2(x)\sum_{k=0}^\infty P_k(x)m^k\nu(dx)\\ =\sum_{k=0}^\infty m^k\int P_2(x) P_k(x)\nu(dx)
   =\sum_{k=0}^{d+1}m^k \int P_2(x)P_k(x)\nu(dx).
    \end{multline}
  Thus,  comparing the right hand sides of \eqref{P2f1} and \eqref{P2f2}  we see
  that
$$
\V(m)= 1-m^2+\V'(0)m+\sum_{k=0}^{d+1}c_km^k
$$
 is a polynomial of degree at most $d+1$, where  $c_k=\int P_2(x)P_k(x)\nu(dx)$.

(iii) We now  prove the converse claim. If $\V$ is a polynomial of degree $d+1$, then recursion \eqref{P-rec} becomes  \eqref{CCRd}.
Proposition \ref{Lem-intP}(iii)  gives $\int P_n(x)\nu(dx)=0$ for $n\geq1$. Noting that
$\{P_j(x):j\leq k\}$ span the same subspace as monomials, to prove $(\nu;d)$-orthogonality it remains to verify that
\begin{equation}
  \label{d-orthogonality-ind}
\int x^k P_n(x)\nu(dx)=0 \mbox{ for all $n\geq 2+(k-1)d$}
\end{equation}
for all $k\in\NN$.

 The proof proceeds by induction on $k$. Consider first the case $k=1$. From \eqref{CCRd} we see that $xP_n$
is a linear combination of $ P_{n+1},P_n,\dots,P_{n-d}$. Thus,   $\int x
P_n(x)\nu(dx)=0$ if $n\geq d+1$. If $n=2,\dots d$, then \eqref{CCRd} shows that $xP_n$ is a linear combination of
$P_{n+1},\dots,P_1$, thus $\int x P_n(x)\nu(dx)=0$, too.

Suppose now that \eqref{d-orthogonality-ind} holds for some $k\geq 1$. Take $n\geq 2+k d$. Then $n>d+1$ so from \eqref{CCRd} we
see that polynomial $x^{k+1}P_n(x)$ is a linear combination of polynomials $\{x^kP_{j}(x): j=n-d,n-d+1,\dots,n+1\}$. Since
$j\geq n-d\geq 2+kd-d=2+(k-1)d$, each of the polynomials $x^kP_{j}(x)$  in the linear combination satisfies the inductive
assumption, $\int x^kP_{j}(x)\nu(dx)=0$. Thus $\int x^{k+1}P_n(d)\nu(dx)=0$, proving that \eqref{d-orthogonality-ind}
holds for all $k\in\NN$.

\end{proof}
Combining the above results with Theorem \ref{prop:Wojtek} we have the following, compare \cite[Th\'{e}or\`eme 2.1]{maroni1989orthogonalite} and  \cite[Theorem 3.1]{Iseghem1987approx},
who study polynomials given by finite recursions under regularity conditions which fail in the case we are interested in.
(Ref.
\cite{da1999shohat} gives a nice introduction to their theory.)
\begin{corollary}\label{Cor-favard3}
Consider polynomials $\{P_n(x)\}$ given by the 4-step recursion:
\begin{eqnarray*}
x P_1(x)&=&P_2(x)+aP_1(x)+P_0(x),\\
x P_2(x)&=&P_3(x)+a P_2(x)+bP_1(x),\\
  x P_n(x)&=& P_{n+1}(x)+a P_{n}(x)+bP_{n-1}(x)+c P_{n-2}(x),\; n\geq 3,
\end{eqnarray*}
with $P_0(x)=1, P_1(x)=x$. Then the following conditions are equivalent.
\begin{enumerate}
  \item $b^3\geq 27 c^2$.
  \item Polynomial $\{P_n\}$ are $(\nu;2)$-orthogonal for some probability measure $\nu$ (which then necessarily has mean 0, variance 1, and compact support).
\end{enumerate}
\end{corollary}
\begin{proof}
  If $b^3\geq 27 c^2$, then by Theorem \ref{prop:Wojtek}, $\V(m)=1+am+(b-1)m^2+cm^3$ is a variance function, and \eqref{f2P} holds. So
  Proposition \ref{T:V2d}(iii) implies (ii).
  Conversely, if (ii) holds, then by Proposition \ref{Lem-intP}(ii),  $\V(m)=1+am+(b-1)m^2+cm^3\in\mathcal{V}$, so Theorem \ref{prop:Wojtek} implies (i).
\end{proof}
\subsection{Proof of  Theorem \ref{T:poly}}\label{ProofofT1.5}
By Proposition \ref{Lem-intP}, for a family of monic polynomials $\{P_n(x)\}$, recursion \eqref{CCRd} holds if and only if
  the generating function \eqref{f2P}
is given by \eqref{f(x,m)} with
\begin{equation}
  \label{rec2V}
  \V(m)=1+b_1 m + (b_2-1)m^2+\sum_{k=3}^{d+1}b_km^k.
\end{equation}
Thus statements (i) and (ii) are equivalent.

Proposition  \ref{T:V2d}(ii) gives  implication (iii)$\Rightarrow$(i), as it says that already  a special case of $(\nu;d)$-orthogonality implies (i);
 the implication (i)$\Rightarrow$(iii) is  %
 Proposition  \ref{T:V2d}(iii).

\section{Additional results and comments}\label{Sec:conclusions}

\subsection{A combinatorial example}\label{Sect:Combinatorics}
Consider the probability distribution on $[0,\infty)$, which in \cite{mlotkowski2010fuss} was denoted $\mu(3,1)$.
Its moments are $\frac{1}{3n+1}\binom{3n+1}{n}$ (Fuss numbers of order 3, A001764 in OEIS)
and the moment generating function, denoted $\mathcal{B}_3(z)$, is
\[
\mathcal{B}_3(z)=\frac{3}{3-4\sin^2\alpha}=\frac{2\sin\alpha}{\sqrt{3z}},
\]
where $\alpha=\frac{1}{3}\arcsin\sqrt{27z/4}$.
The first expression was obtained in \cite{mlotkowskipenson2014},
the second can be obtained by elementary manipulations.
The density function was described in \cite{pensonsolomon,mlotkowskipensonzyczkowski2013}.
We are going to study a probability distribution which is a transformation of $\mu(3,1)$.

\begin{proposition}
If $\mu$ is a probability measure on $[0,\infty)$,
with the moment generating function $M_{\mu}(z)$,
then there exists a probability measure $\mu_1$ on $[0,\infty)$
such that $M_{\mu_1}(z)=\frac{1}{1-z M_{\mu}(z)}$.
\end{proposition}

\begin{proof}
This is a consequence of Proposition~6.1 in %
\cite{bercovici1993free}   with $\psi(z)=\frac{z M_{\mu}(z)}{1-z M_{\mu}(z)}$.
Namely, since $M_{\mu}(z)$ %
is $\mathbb{C}^{+}\to\mathbb{C}^{+}$,
the function
\[
\frac{z}{1-z M_{\mu}(z)}=\frac{z-|z|^2 \overline{M_{\mu}(z)}}{\left|1-z M_{\mu}(z)\right|^2}
\]
is also $\mathbb{C}^{+}\to\mathbb{C}^{+}$.
\end{proof}

Let $\mu$ denote the probability measure which satisfies
\[
M_{\mu}(z)=\frac{1}{1-z\mathcal{B}_3(z)}=\frac{3}{3-2\sqrt{3z}\sin\alpha},
\]
$\alpha=\frac{1}{3}\arcsin\sqrt{27z/4}$.
This identity implies that moments $s(n)$ of $\mu$ satisfy the following recurrence
relation: $s(0)=1$ and for $n\ge1$
\[
s(n)=\sum_{i=0}^{n-1}\frac{1}{3i+1}\binom{3i+1}{i}s(n-1-i).
\]
This sequence appears in OEIS as~A098746:
\[
1, 1, 2, 6, 23, 102, 495, 2549, 13682, 75714, 428882,\ldots
\]
and counts permutations which avoid patterns $4231$ and $42513$,
see \cite{albert2004,martinez2016}.
For $n\ge1$ we have also
\[
s(n)=\sum_{i=0}^{n}\frac{n-i}{n+2i}\binom{n+2i}{i}.
\]

From the equation $\mathcal{B}_3(z)=1+z\mathcal{B}_3(z)^3$ (see \cite{gkp}) we obtain
identity
\begin{equation}\label{momgenfunequation}
zM_{\mu}(z)^2\big(M_{\mu}(z)-1\big)=z^2M_{\mu}(z)^3+\big(M_{\mu}(z)-1\big)^3,
\end{equation}
which yields the free $S$-transform
\[
S_{\mu}(z)=\frac{1+z+\sqrt{(1+z)(1-3z)}}{2(1+z)}.
\]
Substituting $zM_{\mu}(z)\mapsto z$ in (\ref{momgenfunequation})
and applying  (\ref{R2M})   %
we get
\[
z\big(R_{\mu}(z)+1\big)R_{\mu}(z)=z^2\big(R_{\mu}(z)+1\big)+R_{\mu}(z)^3.
\]
Putting $R_{\mu}(z)=z r_{\mu}(z)$ yields
\begin{equation}\label{rtransformequation1}
r_{\mu}(z)-1=zr_{\mu}(z)\big(1-r_{\mu}(z)+r_{\mu}(z)^2\big).
\end{equation}
This implies that $r_{\mu}(z)$ is the generating function
for the sequence A106228:
\[
1, 1, 2, 6, 21, 80, 322, 1347, 5798, 25512, 114236, 518848,\ldots,
\]
which counts Motzkin paths of a special kind.
These are free cumulants of $\mu$, namely $\kappa_n(\mu)=A106228(n-1)$ for $n\ge1$.
Note that the shifted sequence
\[
1, 2, 6, 21, 80, 322, 1347, 5798, 25512, 114236, 518848,\ldots
\]
is not positive definite, %
for example $\det\left(\kappa_{i+j+2}(\mu)\right)_{i,j=0}^{5}=-3374$, so
$\mu$
is not $\boxplus$-infinitely divisible, see \cite{nicaspeicher2006}.

From \eqref{rtransformequation1} one can read out that centered measure
$\nu$ with $r_{\nu}(z)=r_{\mu}(z)-1$
(so that $\nu$ is the translation of $\mu$ by $-1$) has
$$\V_{\nu}(z)=1+2z+2z^2+z^3$$
 and the comment above (or Theorem \ref{prop:Wojtek}) shows that
$\V_{\nu}\not\in\mathcal{V}_\infty$.

\subsection{More on generating functions}
Several authors considered families of   polynomials $\{T_n\}$ with the generating function of the form
\begin{equation}\label{csk type}
\displaystyle\sum_{n=0}^\infty T_n(x)z^n=\frac{M(z)}{N(z)-z x},
\end{equation}
where  $z\mapsto M(z)$ and $z\mapsto N(z)$ are analytic functions in the neighborhood of $0\in\CC$ with $M(0)=N(0)\ne 0$.
See \cite[Lemma 2]{Anshelevich01} with his $u(z)=z/M(z)$ and $f(z)=N(z)/z$  or the generating function in \cite[(3.10)]{Fakhfakh2017}. %
 (See also \cite{bozejko-demni2009generating,Kubo-Kuo-2007}, and the discussion in \cite{bozejko-demni2008Meixner}.)

At first sight \eqref{csk type} looks more general than \eqref{f2P}, but in fact the difference is superficial. The following
result was inspired by results in \cite[Section 3.2]{Fakhfakh2017}.

\begin{proposition}\label{TH1}
Let $\nu$ be a non-degenerate compactly supported probability measure   with mean $0$. Suppose that
the sequence of polynomials $\{T_n\}$ has generating function \eqref{csk type},  $\int T_n(x)\nu(dx)=0$ for $n\geq 1$, and $\int T_n(x)T_1(x)\nu(dx)=0$ for $n\geq 2$. Let $\V$ be the variance function of the \CSK family generated by $\nu$.

Then, with $t=\V(0)/M(0)$ we have $M(z)=\V(tz)/t$ and $N(z)=\V(zt)/t+t z^2$.
In particular, $T_n(x)=t^n P_n(x)$ for all $n=0,1,2\dots$, where the sequence $\{P_n\}$ is given by expansion \eqref{f2P} for the density of the \CSK family generated by $\nu$.
\end{proposition}
(Polynomials $\{P_n\}$ are monic if the variance of $\nu$ is $1$.)

We remark that if in addition,  $\int T_2(x)T_n(x)\nu(dx)=0$ for $n\geq d+2$, then  by Proposition \ref{T:V2d} the variance
function of the \CSK family generated by $\nu$ is a polynomial of degree at most $d+1$. When $d=1$, this  recovers
\cite[Corollary 3.6]{Fakhfakh2017}.
For related results with exponential rather than Cauchy generating functions see \cite{kokonendji2005characterizations,varma2016characterization}.

In order to be able to integrate the series term by term,  we  first confirm that the series  converges uniformly over $x$ from any
compact set. (Compare Proposition \ref{Lem-intP}(i).)

\begin{lemma}\label{Lem-intPA}
Fix $M>0$. Then there is $r>0$ such that the series \eqref{csk type}  converges for all $|x|<M$ and all $|m|<r$.
\end{lemma}
\begin{proof}
The $x$-dependent radius
$r(x)$ of convergence of the series is the minimum modulus  root of equation $ N(z)-zx=0$. Since $N(0)\ne 0$ it is clear
that for every $M>0$ there is $r>0$ such that $|N(z)|>|zx|$ for all $|z|<r$ and all $|x|<M$. So there are no roots in the
disk $|z|<r$ and the radius of convergence is at least $r$.
\end{proof}

\begin{proof}[Proof of Proposition \ref{TH1}]

Choose $r>0$ such that the series \eqref{csk type} converges for all $x$ from the support of $\nu$.   Integrating term-by-term with respect to $\nu$ we get
\begin{multline*}
\int\left(\sum_{n=0}^\infty T_n(x)z^n\right)\nu(dx)= \sum_{n=0}^\infty \int T_n(x) z^n \nu(dx)=\int T_0(x)\nu(dx)
=\int\frac{M(0)} {N(0)}\nu(dx)=1.
\end{multline*} We therefore get
\begin{equation}
  \label{V-H}\int \frac{M(z)}{N(z) -zx}\nu(dx)=1
\end{equation}
for all real $z$ close enough to $0$.

Using this and \eqref{csk type} we  compute
$T_1(x)=(x+M'(0)-N'(0))/M(0)$. Since $\int T_1(x)\nu(dx)=0$, we see that $M'(0)=H'(0)$ and $T_1(x)=\alpha x $ with $\alpha=1/M(0)\ne0$.

  Since $T_1(x)$ is bounded on the support of $\nu$ and the series converges uniformly,   integrating term by term we get
\begin{multline}\label{TnT1}
\int\left(\sum_{n=0}^\infty T_n(x)T_1(x)z^n\right)\nu(dx)=\sum_{n=0}^\infty\int T_n(x)T_1(x)z^n \nu(dx)
\\=z\int T_1^2(x) \nu(dx)  =\alpha^2 \V(0) z,
\end{multline}
 where $\V(0)>0$ is the variance of $\nu$ (recall that $\nu$ is non-degenerate).
On the other hand, using partial fractions we get
\begin{multline}\label{MT1}
\int \frac{M(z)}{N(z)-z x}T_1(x)\nu(dx)=\alpha\int M(z) \left(\frac{  N(z)}{z (N(z)-x z)}-\frac{1 }{z}\right)\nu(dx)\\
=\frac{  \alpha N(z)}{z}\int  \frac{ M(z)}{N(z)-x z} \nu(dx)-\frac{ \alpha M(z) }{z}\int 1 \nu(dx)= \alpha\frac{ N(z)- M(z)}{z}.
\end{multline}
(Here, we used \eqref{V-H} and the fact that $\nu$ is a probability measure.)
Therefore, with $t= \alpha\V(0)=\V(0)/M(0)\ne 0$,  since \eqref{TnT1} and \eqref{MT1} are equal we get $N(z)=M(z)+t z^2$, and \eqref{V-H} takes the form
$$
\int \frac{M(z)}{M(z)+tz^2-z x} \nu(dx)=1.
$$
Substituting $z=m/t$ and setting $\V(m)=tM(m/t)$ we see that
$$
\int \frac{\V(m)}{\V(m)+m(m-x)} \nu(dx)=1.
$$
This shows that  $\V(m)=tM(m/t)$  is the variance function of the \CSK family generated by $\nu$, and it defines  the  corresponding polynomials $\{P_n\}$ via \eqref{f2P} .

To relate polynomials $T_n$  and polynomials $P_n$  we use the above identities to re-write \eqref{csk type} as follows:
$$
\sum_{n=0}^\infty \frac{T_n(x)}{t^n}m^n=\frac{M(m/t)}{N(m/t)-mx/t}=\frac{\V(m)}{\V(m)+m(m-x)}=\sum_{n=0}^\infty P_n(x) m^n.
$$
\end{proof}

\subsection{Sharpness of some results}

\begin{remark} %
\cite[Corollary 2.5]{chistyakov2011characterization} implies sharp results about general quartic polynomials. For example, one can deduce that   $1+ a z^4\in\mathcal{V}$ if and only if  $-1\leq 12a\leq 3$.
\end{remark}
\begin{remark}\label{R:sharp2}
  Theorem \ref{Prop:operations}\eqref{P-W-4}   does not extend to
  $\V_1,\V_2\in\mathcal{V}$. To see this, consider $\V_1(z)=\V_2(z)=1+z^3/6$, which is in $\mathcal{V}$  by Theorem \ref{prop:Wojtek}.
Applying the  operation $\V_1+\V_2-1$  twelve times, we'd get $1+2z^3\in\mathcal{V}$. The latter is not possible.
 Using recursion \eqref{CCRd} and Proposition \ref{T:V2d}(i), one can compute low order moments of the  measure corresponding to the variance function $1+c z^3$.
  The first six moments are  $ (m_1,\dots,m_6)=(0, 1, 0, 2, c, 5)$.
The $4\times 4$ Hankel determinant of these moments is $1-c^2$,  so  $1+2z^3$ is not a variance function.

\end{remark}
\begin{remark}
   Theorem \ref{Prop:operations}\eqref{P:V-z^2}   does not extend to $\V_1\in\mathcal{V}$. To see this, consider   $\V_1(z)=1+4 z^2+2 z^3$, which is in $\mathcal{V}$ by Theorem \ref{prop:Wojtek},
and apply the operation 4 times to get  $1+2z^3$ which is not in $\mathcal{V}$, as was already noted in Remark \ref{R:sharp2}.
\end{remark}

{\bf Acknowledgement}
The authors thank  Takahiro Hasebe and Kamil Szpojankowski for helpful discussions. %
 W{\l}odzimierz
 Bryc's research was supported in part by the Charles Phelps Taft Research Center at the University of Cincinnati.
 Wojciech M{\l}otkowski is supported by NCN grant 2016/21/B/ST1/00628.


\end{document}